\documentclass[12pt]{article}
\usepackage{amssymb,amsmath,amsthm}
\usepackage[mathscr]{euscript}
\usepackage[left=3cm,right=3cm]{geometry}
\title{An Entourage Approach to the Contraction Principle in Uniform Spaces Endowed with a Graph\\[0.3cm]}
\author{{Aris Aghanians$^1$,\,\,\,Kamal Fallahi$^1$,\,\,\,Kourosh Nourouzi$^{1}$\thanks{Corresponding
author } \thanks {e-mail: nourouzi@kntu.ac.ir; fax: +98 21
22853650}}\\[0.4cm]
{\em $^1$Department of Mathematics, K. N. Toosi University of Technology,}\\
{\em P.O. Box 16315-1618, Tehran, Iran}}
\newtheorem{defn}{Definition}
\newtheorem{cor}{Corollary}
\newtheorem{exm}{Example}
\newtheorem{thm}{Theorem}
\newtheorem{prop}{Proposition}

\newtheorem{lem}{Lemma}

\newcommand{\fix}{{\rm Fix}}
\newcommand{\card}{{\rm card}}
\begin{document}
\maketitle \begin{abstract} In this paper, we study Banach
contractions in uniform spaces endowed with a graph and give some
sufficient conditions for a mapping to be a Picard operator. Our
main results generalize some results of [J. Jachymski, ``The
contraction principle for mappings on a metric space with a graph",
Proc. Amer. Math. Soc. 136 (2008) 1359-1373] employing the basic entourages of the uniform space.
\end{abstract}
\def\thefootnote{ \ }
\footnotetext{{\em}$2010$ Mathematics Subject Classification.
47H10, 05C40.\par {\bf Keywords:} Separated uniform space;
Weakly connected graph; Banach $G$-contraction; Fixed point;
Picard operator.}

\section{Introduction and Preliminaries}
In \cite{ach}, Acharya investigated Banach, Kannan and
Chatterjea type contractions in uniform spaces using the basic
entourages and gave some sufficient conditions for a mapping to
have a unique fixed point. Recently, Jachymski \cite{jac}
established  some fixed point theorems for Banach $G$-contractions in
metric spaces endowed with a graph. His results generalized the
Banach contraction principle in both metric and partially ordered
metric spaces. In \cite{agh}, the authors generalized Theorems 3.1, 3.2, 3.3 and 3.4 as well as Proposition 3.1 of \cite{jac} from metric to uniform spaces endowed with a graph using $\mathscr E$-distances under  weaker contractive conditions.\par
In this paper, combining Acharya's and Jachymski's ideas we aim to generalize the above-mentioned theorems as well as Theorem 3.5 and Proposition 2.1 of \cite{jac} from metric to uniform spaces endowed with a graph using the basic entourages in a completely different manner than one proposed in \cite{agh}.\par
We start by reviewing a few basic notions in uniform
spaces. An in-depth discussion may be found in many standard
texts, e.g., \cite[pp. 238-277]{wil}.\par Let $X$ be an arbitrary
nonempty set. A nonempty collection $\mathscr U$ of subsets of
$X\times X$ (called the entourages of $X$) is called a uniformity
on $X$ if
\begin{description}
\item U1) each member of $\mathscr U$ contains the diagonal
$\Delta(X)=\{(x,x):x\in X\}$; \item U2) if $U,V\in\mathscr U$, then so is $U\cap V$; \item U3) for each entourage $U$ of
$X$, the set $\{(x,y):(y,x)\in U\}$ is again an entourage of $X$;
\item U4) for each entourage $U$ of $X$, there exists another
entourage $V$ of $X$ such that $V\circ V\subseteq U$; \item U5)
the superset of every entourage of $X$ belongs to  $\mathscr U$.
\end{description}
If $\mathscr U$ satisfies (U1)-(U5), then the pair $(X,\mathscr
U)$ (shortly denoted by $X$) is called a uniform space.\par A
uniform space $X$ is separated if the
intersection of all entourages of $X$ coincides with the
diagonal $\Delta(X)$. If this is satisfied, then $X$ is called a
separated uniform space.\par A uniformity $\mathscr U$ induces  a
topology
$$\tau_{\mathscr U}=\big\{A\subseteq X:\forall x\in A\ \ \exists U\in\mathscr
U;\ U[x]\subseteq A\big\},$$ on $X$ where $U[x]=\{y\in X:(x,y)\in
U\}$ for all $U\in\mathscr U$ and all $x\in X$. The topology
$\tau_{\mathscr U}$ is called the uniform topology on $X$ and the
family $\{U[x]:U\in\mathscr U\}$ forms a neighborhood base at
$x\in X$ in the uniform topology.\par A net
$\{x_\lambda:\lambda\in(\Lambda,\geqslant)\}$ in $X$ is said to be
convergent to a point $x\in X$, denoted by $x_\lambda\rightarrow
x$, if it converges to $x$ in the uniform topology $\tau_{\mathscr
U}$, that is, for each entourage $U\in\mathscr U$, there exists a
$\lambda_0\in\Lambda$ such that $(x_\lambda,x)\in U$ for all
$\lambda\geqslant\lambda_0$, and it is said to be Cauchy if for
each entourage $U\in\mathscr U$, there exists a
$\lambda_0\in\Lambda$ such that $(x_\lambda,x_\mu)\in U$ for all
$\lambda,\mu\geqslant\lambda_0$. The uniform space $X$ is called
complete if each Cauchy net in $X$ is convergent to some point of
$X$ and sequentially complete if each Cauchy sequence in $X$ is
convergent to some point of $X$. Also, a mapping $T$ from $X$ into
itself is continuous on $X$ if and only if $x_\lambda\rightarrow
x$ implies $Tx_\lambda\rightarrow Tx$ for all nets
$\{x_\lambda:\lambda\in\Lambda\}$ and all points $x$ in $X$.\par
For any pseudometric $\rho$ on $X$ and any $r>0$, set
$$V(\rho,r)=\big\{(x,y)\in X\times X:\rho(x,y)<r\big\}.$$
Let $\mathscr F$ be a family of (uniformly continuous)
pseudometrics on $X$ that generates the uniformity $\mathscr U$
(see, \cite[Theorem 2.1]{ach}), and denote by $\mathscr V$, the
collection of all sets of the form
$$\bigcap_{i=1}^mV(\rho_i,r_i),$$
where $m$ is a positive integer, $\rho_i\in\mathscr F$ and $r_i>0$
for $i=1,\ldots,m$. Then it is well-known that $\mathscr V$ is a
base for the uniformity $\mathscr U$, i.e., each entourage of $X$
contains at least one member of $\mathscr V$. The elements of
$\mathscr V$ are called the basic entourages of $X$. Now, if
$$V=\bigcap_{i=1}^mV(\rho_i,r_i)\in\mathscr V$$ and $\alpha>0$, then
the set
$$\alpha V =\bigcap_{i=1}^mV(\rho_i,\alpha r_i)$$
is still an element of $\mathscr V$.\par The next lemma, embodies
three important properties of the basic entourages. For more other
properties, the reader may refer to \cite[Lemmas 2.1-2.6]{ach}.

\begin{lem}{\rm\cite{ach}}\label{lem1} Let $X$ be a uniform space
and $V$ be a basic entourage of $X$.
\begin{description}
\item[\rm i)] If $0<\alpha\leq\beta$, then $\alpha V\subseteq\beta
V$. \item[\rm ii)] If
$\alpha,\beta>0$, then $\alpha V\circ\beta
V\subseteq(\alpha+\beta)V$. \item[\rm iii)] For each $x,y\in X$ there exists a positive
number $\lambda$ such that $(x,y)\in\lambda V$. \item[\rm iv)] There exists a pseudometric $\rho$ on $X$ (called the Minkowski's pseudometric of $V$) such that $V=V(\rho,1)$.
\end{description}
\end{lem}

Now, we review a few basic notions of graph theory. For more
details, it is referred to \cite{bon}.\par Consider a directed
graph $G$ with $V(G)=X$ and $E(G)\supseteq\Delta(X)$, that is,
$E(G)$ contains all loops, and let $G$ have no parallel edges.
Denoted by $\widetilde G$, the undirected graph obtained from $G$
by ignoring the directions of the edges of $G$, that is,
$$V(\widetilde G)=X,\quad E(\widetilde G)=\Big\{(x,y)\in X\times
X:(x,y)\in E(G)\vee (y,x)\in E(G)\Big\}.$$ \indent If $x,y\in X$,
then by a path in $G$ from $x$ to $y$, it is meant a finite
sequence $(x_i)_{i=0}^N$ consisting of $N+1$ vertices of $G$ such
that $x_0=x$, $x_N=y$, and $(x_{i-1},x_i)$ is an edge of $G$ for
$i=1,\ldots,N$. A graph $G$ is said to be connected if there
exists a path in $G$ between every two vertices of $G$, and weakly
connected if the graph $\widetilde G$ is connected.\par A subgraph
of $G$ is a graph $H$ such that $V(G)$ and $E(G)$ contain $V(H)$
and $E(H)$, respectively, and that $(x,y)\in E(G)$ implies $x,y\in
V(H)$ for all $x,y\in X$.\par If $x\in X$ and the set $E(G)$ is
symmetric, then the subgraph $G_x$ consisting of all edges and
vertices of $G$ that are contained in a path in $G$ starting at
$x$ is called the component of $G$ containing $x$. So
$V(G_x)=[x]_G$, where $[x]_G$ is the equivalence class of $x$ in
the equivalence relation $\sim$ defined as following:
$$y\sim z\ \Longleftrightarrow\ \mbox{there exists a path in
}G\mbox{ from }y\mbox{ to }z\qquad(y,z\in X).$$ Clearly, the graph
$G_x$ is connected for all $x\in X$.

\section{Main Results}
Throughout this section, $X$ is supposed to be a nonempty set
equipped with a separating uniformity $\mathscr U$ and a directed
graph $G$ such that $V(G)=X$ and $E(G)\supseteq\Delta(X)$ unless
otherwise stated. Moreover, we assume that $\mathscr{F}$ is a
nonempty family of (uniformly continuous) pseudometrics on $X$
generating the uniformity $\mathscr{U}$, and $\mathscr{V}$ is the
collection of all sets of the form $\bigcap_{i=1}^m
V(\rho_i,r_i)$, where $m$ is a positive integer,
$\rho_i\in\mathscr{F}$ and $r_i>0$ for $i=1,\ldots,m$. We denote
here the set of all fixed points of a mapping $T:X\rightarrow X$
by $\fix(T)$ and as usual we put $X_T=\{x\in X:(x,Tx)\in
E(G)\}$.\par We begin by the definition of Banach $G$-contractions
using the basic entourages of $X$, whose idea is taken from
\cite[Definition 2.1]{jac}.

\begin{defn}\label{gcontraction}\rm We say that a mapping $T:X\rightarrow X$ is a
Banach $G$-contraction if
\begin{description}
\item[B1)] $T$ preserves the edges of $G$, i.e, $(x,y)\in E(G)$
implies $(Tx,Ty)\in E(G)$ for all $x,y\in X$; \item[B2)] $T$
decreases the weights of the edges of $G$ in the sense that there
exists an $\alpha\in(0,1)$ such that $(x,y)\in V\cap E(G)$ implies
$(Tx,Ty)\in\alpha V$ for all $x,y\in X$ and all $V\in\mathscr V$.
\end{description}
The number $\alpha$ is called the contractive constant of $T$.
\end{defn}

We give a few  examples of Banach $G$-contractions in uniform
spaces endowed with a graph.

\begin{exm}\rm
\begin{enumerate}
\item Since $E(G)$ and each basic entourage of $X$ contains the
diagonal $\Delta(X)$, it follows that each constant mapping
$T:X\rightarrow X$ is a Banach $G$-contraction with any
contractive constant $\alpha\in(0,1)$. \item Let $G_0$ be the
complete graph with $V(G_0)=X$, i.e., $E(G_0)=X\times X$. Then
Banach $G_0$-contractions are precisely the contractive mappings
on $X$, that is, mappings $T:X\rightarrow X$ for which there
exists an $\alpha\in(0,1)$ such that $(x,y)\in V$ implies
$(Tx,Ty)\in\alpha V$ for all $x,y\in X$ and all $V\in\mathscr V$.
The existence of fixed points for these contractions and the
convergence of their sequences of Picard iterations on sequentially
complete uniform spaces were investigated by Acharya \cite[Theorem
3.1]{ach}. \item Let $\preceq$ be a partial order on $X$, and
consider the graphs $G_1$ and $G_2$ with $V(G_1)=V(G_2)=X$,
$$E(G_1)=\big\{(x,y)\in X\times X:x\preceq y\big\},$$
and
$$E(G_2)=\big\{(x,y)\in X\times X:x\preceq y\ \vee\ y\preceq x\big\}.$$
Then $E(G_1)$ and $E(G_2)$ contain all loops. Now, Banach
$G_1$-contractions are precisely the nondecreasing order
contractive mappings on $X$, that is, mappings $T:X\rightarrow X$
for which $x\preceq y$ implies $Tx\preceq Ty$ for all $x,y\in X$,
and $x\preceq y$ and $(x,y)\in V$ imply $(Tx,Ty)\in\alpha V$ for
all $x,y\in X$ and all $V\in\mathscr V$, where $\alpha\in(0,1)$ is
constant. In fact, a mapping preserves the edges of $G_1$ if and
only if it is nondecreasing with respect to $\preceq$.
Furthermore, Banach $G_2$-contractions are precisely the order
contractive mappings on $X$ which map comparable elements onto
comparable elements of $X$.
\end{enumerate}
\end{exm}

Since the basic entourages of $X$ are all symmetric, the next
proposition is an immediate consequence of Definition
\ref{gcontraction}.

\begin{prop}\label{tilde} If a mapping from  $X$ into itself satisfies Condition {\rm(B1)}
(respectively, Condition {\rm(B2)}) for a graph $G$, then it
satisfies Condition {\rm(B1)} (respectively, Condition {\rm(B2)})
for the graphs $G^{-1}$ and $\widetilde G$. In particular, all
Banach $G$-contractions are both Banach $G^{-1}$- and also Banach
$\widetilde G$-contractions.
\end{prop}

In addition to Lemma \ref{lem1}, we need the following lemma to
prove our results:

\begin{lem}\label{jaclemma}
Let $T:X\rightarrow X$ be a Banach $G$-contraction with
contractive constant $\alpha$. Then for each $x\in X$, each
$y\in[x]_{\widetilde G}$, and each $V\in\mathscr V$, there exists
an $r(x,y;V)>0$ such that
$$(T^nx,T^ny)\in\alpha^nr(x,y;V)V,$$
for all $n\geq1$.
\end{lem}

\begin{proof} Let $x\in X$, $y\in[x]_{\widetilde G}$, and $V\in\mathscr V$ be given. Then
there is a path $(x_i)_{i=0}^N$ in $\widetilde G$ from $x$ to $y$,
i.e., $x_0=x$, $x_N=y$, and $(x_{i-1},x_i)\in E(\widetilde G)$ for
$i=1,\ldots,N$. Since $T$ is a Banach $G$-contraction, it follows
by Proposition \ref{tilde} that $T$ is a Banach $\widetilde
G$-contraction and so it preserves the edges of $\widetilde G$.
Hence an easy induction yields
$$(T^nx_{i-1},T^nx_i)\in E(\widetilde G)\qquad(i=1,\ldots,N,\ n\geq1).$$
Now, if $n\geq1$ is arbitrary, then for each $i=1,\ldots,N$, there
exists a positive number $\lambda_i$ such that
$(x_{i-1},x_i)\in\lambda_iV$, and using induction and the Banach
$\widetilde G$-contractivity of $T$ it follows that
$(T^nx_{i-1},T^nx_i)\in\alpha^n(\lambda_iV)$. Hence, in view of
Lemma \ref{lem1}, we have
$$(T^nx,T^ny)=(T^nx_0,T^nx_N)\in(\alpha^n\lambda_1V)\circ\cdots\circ(\alpha^n\lambda_NV)
\subseteq\Big(\alpha^n\sum_{i=1}^N\lambda_i\Big)V.$$ So it
suffices to put
$$r(x,y;V)=\sum_{i=1}^N\lambda_i=\lambda_1+\cdots+\lambda_N>0.$$
\end{proof}

Similar to Jachymski's idea \cite{jac}, we define Cauchy
equivalent sequences in uniform spaces.

\begin{defn}\rm We say that two sequences $\{x_n\}$ and $\{y_n\}$ in
$X$ are Cauchy equivalent whenever
\begin{description}
\item[C1)] $\{x_n\}$ and $\{y_n\}$ are both Cauchy sequences in
$X$; \item[C2)] for each entourage $U$ in $\mathscr U$, there
exists an $N>0$ such that $(x_n,y_n)\in U$ for all $n\geq N$.
\end{description}
\end{defn}

Now, we are ready to prove our first main theorem.

\begin{thm}\label{main1} The following statements are equivalent:
\begin{description}
\item[\rm i)] $G$ is weakly connected; \item[\rm ii)] For each
Banach $G$-contraction $T:X\rightarrow X$ and each $x,y\in X$, the
sequences $\{T^nx\}$ and $\{T^ny\}$ are Cauchy equivalent;
\item[\rm iii)] Each Banach $G$-contraction has at most one fixed
point in $X$.
\end{description}
\end{thm}

\begin{proof}
(i $\Rightarrow$ ii) Let $x,y\in X$ and $T:X\rightarrow X$ be a
Banach $G$-contraction with contractive constant $\alpha$. To see
that the sequence $\{T^nx\}$ is Cauchy in $X$, first note that
since $G$ is weakly connected, we have $[x]_{\widetilde G}=X$ and
so $Tx\in[x]_{\widetilde G}$. Pick a basic entourage $V$ of $X$
and denote by $\rho$ the Minkowski's pseudometric of $V$. Then by Lemma
\ref{jaclemma}, for each positive integer $n$ we have
$(T^nx,T^{n+1}x)\in\alpha^nr(x,Tx;V)V$, and hence
$$\rho(T^nx,T^{n+1}x)<\alpha^nr(x,Tx;V)\qquad(n\geq1).$$
Therefore,
$$\sum_{n=1}^\infty\rho(T^nx,T^{n+1}x)\leq\sum_{n=1}^\infty\alpha^nr(x,Tx;V)
=\frac{\alpha r(x,Tx;V)}{1-\alpha}<\infty.$$ Now, an easy argument
shows that $\rho(T^nx,T^mx)\rightarrow 0$ as
$n,m\rightarrow\infty$. So there exists an integer $N$ such that
$\rho(T^nx,T^mx)<1$, that is, $(T^nx,T^mx)\in V$ for all $n,m\geq
N$. Because $V\in\mathscr V$ was arbitrary, the sequence
$\{T^nx\}$ is Cauchy in $X$. Similarly, one can show that
$\{T^ny\}$ is Cauchy in $X$. To establish Condition (C2), we
observe first that by the weak connectivity of $G$, we have
$y\in[x]_{\widetilde G}$. So, by Lemma \ref{jaclemma}, given any
$V\in\mathscr V$, there exists an $r(x,y;V)>0$ such that
$(T^nx,T^ny)\in\alpha^nr(x,y;V)V$ for all $n\geq1$. Choose $N>0$
sufficiently large so that $\alpha^nr(x,y;V)<1$ for all $n\geq N$.
Hence by Lemma \ref{lem1}, we have
$$(T^nx,T^ny)\in\alpha^nr(x,y;V)V\subseteq V\qquad(n\geq N).$$
Therefore, Condition (C2) is satisfied and the sequences
$\{T^nx\}$ and $\{T^ny\}$ are Cauchy equivalent.\par (ii
$\Rightarrow$ iii) If $x$ and $y$ are two fixed points for a
Banach $G$-contraction $T:X\rightarrow X$, then the sequences
$\{T^nx\}$ and $\{T^ny\}$ are Cauchy equivalent. So for any basic
entourage $V$ of $X$ and sufficiently large $n$ we have
$$(x,y)=(T^nx,T^ny)\in V.$$
Because $V$ was arbitrary and $X$ is separated, we get $x=y$.\par
(iii $\Rightarrow$ i) Suppose on the contrary that $G$ is not
weakly connected. So there exists an $x_0\in X$ such that both
sets $[x_0]_{\widetilde G}$ and $X\setminus[x_0]_{\widetilde G}$
are nonempty. Fix any $y_0\in X\setminus[x_0]_{\widetilde G}$ and
define a mapping $T:X\rightarrow X$ by
$$Tx=\left\{\begin{array}{cc}
x_0&x\in[x_0]_{\widetilde G}\\\\
y_0&x\in X\setminus[x_0]_{\widetilde G}
\end{array}\right..$$
Clearly, $y_0\neq x_0$ and $\fix(T)=\{x_0,y_0\}$. We claim that
$T$ is a Banach $G$-contraction. For, let $(x,y)\in E(G)\subseteq
E(\widetilde G)$. Then $[x]_{\widetilde G}=[y]_{\widetilde G}$;
therefore, both $x$ and $y$ belong to either $[x_0]_{\widetilde
G}$ or $X\setminus[x_0]_{\widetilde G}$, and so $Tx=Ty$. Because
$E(G)$ contains all loops, $T$ preserves the edges of $G$. To
establish Condition (B2), let $(x,y)\in V\cap E(G)$, where $V$ is
an arbitrary element of $\mathscr V$. Then $Tx=Ty$, and so
$$(Tx,Ty)\in\Delta(X)\subseteq\alpha V,$$
for any constant $\alpha\in(0,1)$. Consequently, $T$ is a Banach
$G$-contraction, which is a contradiction. Hence the graph $G$ is
weakly connected.
\end{proof}

By a careful look at the proof of Theorem \ref{main1}, it is
understood that the hypothesis that $X$ is separating is only used
in the proof of (ii $\Rightarrow$ iii). Therefore, the other parts
are still true in arbitrary (not necessarily separated) uniform
spaces. For instance, if each Banach $G$-contraction has at most
one fixed point in an arbitrary uniform space $X$, then the graph
$G$ is weakly connected.\par The next result is a consequence of
Theorem \ref{main1}.

\begin{cor}\label{corollary} If $X$ is sequentially complete, then the following
statements are equivalent:
\begin{description}
\item[\rm i)] $G$ is weakly connected; \item[\rm ii)] For each
Banach $G$-contraction $T:X\rightarrow X$, there exists an $x^*\in
X$ such that $T^nx\rightarrow x^*$ for all $x\in X$.
\end{description}
\end{cor}

\begin{proof} (i $\Rightarrow$ ii) Let $T:X\rightarrow X$ be a
Banach $G$-contraction and let $x\in X$. Since $G$ is weakly
connected, it follows by Theorem \ref{main1} that the sequence
$\{T^nx\}$ is Cauchy in $X$, and since $X$ is sequentially
complete, there exists an $x^*\in X$ such that $T^nx\rightarrow
x^*$. Now, let $y\in X$ be given. Then for any basic entourage
$V$, there exists a $V_0\in\mathscr V$ such that $V_0\circ
V_0\subseteq V$. By Cauchy equivalence of $\{T^ny\}$ and
$\{T^nx\}$, we may pick an integer $N>0$ such that
$$(T^ny,T^nx),(T^nx,x^*)\in V_0\qquad(n\geq N).$$
Therefore,
$$(T^ny,x^*)\in V_0\circ V_0\subseteq V$$
for all $n\geq N$. So $T^ny\rightarrow x^*$.\par (ii $\Rightarrow$
i) Let $T:X\rightarrow X$ be a Banach $G$-contraction. Then the
only possible fixed point of $T$ is $x^*$. Indeed, if $y^*\in X$
is a fixed point for $T$, then $y^*=T^ny^*\rightarrow x^*$. Since
$X$ is separated, it follows that $y^*=x^*$. Hence $T$ has at most
one fixed point in $X$. So by Theorem \ref{main1}, the graph $G$
is weakly connected.
\end{proof}

\begin{prop}\label{proposition} Let $T:X\rightarrow X$ be a Banach $G$-contraction
and $x_0\in X$. If $Tx_0\in[x_0]_{\widetilde G}$, then
$[x_0]_{\widetilde G}$ is $T$-invariant and
$T\mid_{[x_0]_{\widetilde G}}$ is a Banach ${\widetilde
G}_{x_0}$-contraction. Furthermore, the sequences $\{T^nx\}$ and
$\{T^ny\}$ are Cauchy equivalent for all $x,y\in[x_0]_{\widetilde
G}$.
\end{prop}

\begin{proof} Note  because $Tx_0\in[x_0]_{\widetilde G}$, we
have $[Tx_0]_{\widetilde G}=[x_0]_{\widetilde G}$. Suppose that
$x\in[x_0]_{\widetilde G}$. Then there exists a path
$(x_i)_{i=0}^N$ in $\widetilde G$ from $x_0$ to $x$, i.e., $x_N=x$
and $(x_{i-1},x_i)\in E(\widetilde G)$ for $i=1,\ldots,N$. Since
$T$ is a Banach $G$-contraction, it follows by Proposition
\ref{tilde} that $T$ is a Banach $\widetilde G$-contraction and so
it preserves the edges of $\widetilde G$. Hence $(Tx_i)_{i=0}^N$
is a path in $\widetilde G$ from $Tx_0$ to $Tx$. Therefore,
$Tx\in[Tx_0]_{\widetilde G}=[x_0]_{\widetilde G}$ and whence
$[x_0]_{\widetilde G}$ is $T$-invariant.\par Because $T$ is itself
a Banach $\widetilde G$-contraction and $E({\widetilde
G}_{x_0})\subseteq E(\widetilde G)$, to see that it is a Banach
${\widetilde G}_{x_0}$-contraction on $[x_0]_{\widetilde G}$, it
suffices to show that $T$ preserves the edges of ${\widetilde
G}_{x_0}$. To this end, suppose that $(x,y)$ is any edge of
${\widetilde G}_{x_0}$. Then $x,y\in V({\widetilde
G}_{x_0})=[x_0]_{\widetilde G}$ and hence there is a path
$(x_i)_{i=0}^N$ in $\widetilde G$ from $x_0$ to $x$, i.e.,
$x_N=x$, and $(x_{i-1},x_i)\in E(\widetilde G)$ for
$i=1,\ldots,N$. Therefore, setting $x_{N+1}=y$, we see that
$(x_i)_{i=0}^{N+1}$ is a path in $\widetilde G$ from $x_0$ to $y$.
On the other hand, $Tx_0\in[x_0]_{\widetilde G}$. So there is
another path $(y_j)_{j=0}^M$ in $\widetilde G$ from $x_0$ to
$Tx_0$, i.e., $y_0=x_0$, $y_M=Tx_0$ and $(y_{j-1},y_j)\in
E(\widetilde G)$ for $j=1,\ldots,M$. Using the Banach $\widetilde
G$-contractivity of $T$, we see that
$$(x_0=y_0,y_1,\ldots,y_M=Tx_0,Tx_1,\ldots,Tx_N=Tx,Tx_{N+1}=Ty)$$
is a path in $\widetilde G$ from $x_0$ to $Ty$. In particular,
$(Tx,Ty)=(Tx_N,Tx_{N+1})\in E(\widetilde G)$. Moreover,
$Tx,Ty\in[x_0]_{\widetilde G}=V({\widetilde G}_{x_0})$. Thus,
$(Tx,Ty)\in E({\widetilde G}_{x_0})$.\par Finally, since the graph
${\widetilde G}_{x_0}$ is weakly connected, $V({\widetilde
G}_{x_0})=[x_0]_{\widetilde G}$, and $T:[x_0]_{\widetilde
G}\rightarrow[x_0]_{\widetilde G}$ is a Banach ${\widetilde
G}_{x_0}$-contraction, it follows by Theorem \ref{main1} that the
sequences $\{T^nx\}$ and $\{T^ny\}$ are Cauchy equivalent for all
$x,y\in[x_0]_{\widetilde G}$.
\end{proof}

Following the ideas of Petru\c sel and Rus \cite{pet}, and also
Jachymski \cite{jac}, we define Picard and weakly Picard operators
in uniform spaces.

\begin{defn}\rm Let $T$ be a mapping from $X$ into itself. We say
that
\begin{description}
\item[\rm i)] $T$ is a Picard operator if $T$ has a unique fixed
point $x^*$ in $X$ and $T^nx\rightarrow x^*$ as
$n\rightarrow\infty$ for all $x\in X$. \item[\rm ii)] $T$ is a
weakly Picard operator if the sequence $\{T^nx\}$ converges to a
fixed point of $T$ for all $x\in X$.
\end{description}
\end{defn}

It is clear that each Picard operator is a weakly Picard operator
but the converse is not generally true. In fact, the mapping $T$
defined in the proof of (iii $\Rightarrow$ i) in Theorem
\ref{main1} is a weakly Picard operator which fails to be a Picard
operator.\par Now, we are ready to prove our second main result.

\begin{thm}\label{main2} Let $X$ be sequentially complete
and have the following property:
\begin{description}
\item[$(\ast)$] If a sequence $\{x_n\}$ converges to some $x$ in
$X$ and it satisfies $(x_n,x_{n+1})\in E(G)$ for all $n\geq1$,
then there exists a subsequence $\{x_{n_k}\}$ of $\{x_n\}$ such
that $(x_{n_k},x)\in E(G)$ for all $k\geq1$.
\end{description}
Suppose that $T:X\rightarrow X$ is a Banach $G$-contraction. Then
the following assertions hold:
\begin{enumerate}
\item $T\mid_{[x]_{\widetilde G}}$ is a Picard operator for all
$x\in X_T$. \item If $X_T$ is nonempty and $G$ is weakly
connected, then $T$ is a Picard operator. \item
$\card(\fix(T))=\card\{[x]_{\widetilde G}:x\in X_T\}$. \item $T$
has a fixed point in $X$ if and only if $X_T$ is nonempty. \item
$T$ has a unique fixed point in $X$ if and only if there exists an
$x\in X_T$ such that $X_T\subseteq[x]_{\widetilde G}$. \item
$T\mid_{X'}$ is a weakly Picard operator, where
$X'=\bigcup\{[x]_{\widetilde G}:x\in X_T\}$. \item If $X_T=X$,
then $T$ is a weakly Picard operator.
\end{enumerate}
\end{thm}

\begin{proof} We prove each part of the theorem separately.
\begin{enumerate}
\item Let $x\in X_T$. Then $(x,Tx)\in E(G)\subseteq E(\widetilde
G)$ and hence $Tx\in[x]_{\widetilde G}$. Now, by Proposition
\ref{proposition}, the mapping $T\mid_{[x]_{\widetilde G}}$ is a
Banach ${\widetilde G}_x$-contraction and if $y\in[x]_{\widetilde
G}$, then the sequences $\{T^nx\}$ and $\{T^ny\}$ are Cauchy
equivalent. Since $X$ is sequentially complete, an argument
similar to that appeared in the proof of (i $\Rightarrow$ ii) in
Corollary \ref{corollary} establishes that there exists an $x^*\in
X$ such that $T^ny\rightarrow x^*$ for all $y\in[x]_{\widetilde
G}$. In particular, $T^nx\rightarrow x^*$. We claim that $x^*$ is
the unique fixed point of $T$ in $[x]_{\widetilde G}$. To prove
our claim, note first that since $T$ is a Banach $G$-contraction
and $(x,Tx)\in E(G)$, it follows that $(T^nx,T^{n+1}x)\in E(G)$
for all $n\geq1$, and so Property $(\ast)$ implies the existence
of a strictly increasing sequence $\{n_k\}$ of positive integers
with $(T^{n_k}x,x^*)\in E(G)$ for all $k\geq1$. Now,
$(x,Tx,\ldots,T^{n_1}x,x^*)$ is a path in $G$ and hence in
$\widetilde G$ from $x$ to $x^*$, that is, $x^*\in[x]_{\widetilde
G}$.\par To see that $x^*$ is the unique fixed point of $T$, let
$V$ be any basic entourage of $X$ and pick a member $V_0$ in
$\mathscr V$ such that $V_0\circ V_0\circ V_0\subseteq V$. Because
$T^{n_k}x\rightarrow x^*$ and $\{T^{n_k}x\}$ and $\{T^{n_k+1}x\}$
are Cauchy equivalent, we may take an integer $k\geq1$
sufficiently large so that
$$(T^{n_k}x,x^*),(T^{n_k}x,T^{n_k+1}x)\in V_0\cap E(G).$$
Because $T$ is a Banach $G$-contraction, we have
$$(T^{n_k+1}x,Tx^*)\in\alpha V_0\subseteq V_0,$$ where
$\alpha$ is the contractive constant of $T$. Therefore, by the
symmetry of $V_0$, we get
$$(x^*,Tx^*)\in V_0\circ V_0\circ
V_0\subseteq V.$$ Since $V\in\mathscr V$ was arbitrary and $X$ is
separated, it follows that $x^*=Tx^*$.\par Finally, if $y^*$ is a
fixed point for $T$ in $[x_0]_{\widetilde G}$, then
$y^*=T^ny^*\rightarrow x^*$ and so by the uniqueness of the limits
of convergent sequences in separated uniform spaces, we have
$y^*=x^*$. Consequently, $T\mid_{[x]_{\widetilde G}}$ is a Picard
operator. \item If $x\in X_T$, since $G$ is weakly connected, it
follows that $[x]_{\widetilde G}=X$. So by 1, $T$ is a Picard
operator. \item Put ${\cal C}=\{[x]_{\widetilde G}:x\in X_T\}$ and
define
$$\Gamma:\fix(T)\rightarrow \cal C$$
$$\Gamma(x)=[x]_{\widetilde G}.$$
We are going to show that $\Gamma$ is a bijection. Suppose first
that $x\in\fix(T)$. Then $(x,Tx)=(x,x)\in E(G)$. So $x\in X_T$ and
$\Gamma(x)\in\cal C$. Moreover, $x_1=x_2$ implies
$\Gamma(x_1)=\Gamma(x_2)$ for all $x_1,x_2\in\fix(T)$. Hence the
mapping $\Gamma$ is well-defined.\par To see that $\Gamma$ is
surjective, let $x$ be any point of $X_T$. Since by 1,
$T\mid_{[x]_{\widetilde G}}$ is a Picard operator, it has a unique
fixed point in $[x]_{\widetilde G}$, say $x^*$. Now we have
$\Gamma(x^*)=[x^*]_{\widetilde G}=[x]_{\widetilde G}$.\par
Finally, if $x_1$ and $x_2$ are two fixed points for $T$ such that
$$[x_1]_{\widetilde G}=\Gamma(x_1)=\Gamma(x_2)=[x_2]_{\widetilde
G},$$ then $x_1\in X_T$, and by 1, $T\mid_{[x_1]_{\widetilde G}}$
is a Picard operator. Therefore, $x_1$ and $x_2$ are two fixed
points for $T$ in $[x_1]_{\widetilde G}$ and because $T$ must have
only one fixed point in $[x_1]_{\widetilde G}$, it follows that
$x_1=x_2$. Hence $\Gamma$ is injective and consequently, it is a
bijection. \item It is an immediate consequence of 3. \item
Suppose first that $x$ is the unique fixed point for $T$. Then
$x\in X_T$ and by 3, for any $y\in X_T$, we have $[y]_{\widetilde
G}=[x]_{\widetilde G}$. So $y\in[x]_{\widetilde G}$.\par For the
converse, note first that since $X_T$ is nonempty, it follows by 4
that $T$ has at least one fixed point in $X$. Now, if
$x^*,y^*\in\fix(T)$, then $x^*,y^*\in X_T\subseteq[x]_{\widetilde
G}$ and so $[x^*]_{\widetilde G}=[y^*]_{\widetilde
G}=[x]_{\widetilde G}$. Consequently, the one-to-one
correspondence in 3 implies that $x^*=y^*$. \item If
$X_T=\emptyset$, then so is $X'$ and vice versa, and there is
nothing to prove. So let $x\in X'$. Then there exists an $x_0\in
X_T$ such that $x\in[x_0]_{\widetilde G}$. Since, by Assertion 1,
$T\mid_{[x_0]_{\widetilde G}}$ is a Picard operator, it follows
that the sequence $\{T^nx\}$ converges to a fixed point of $T$.
Therefore, $T\mid_{X'}$ is a weakly Picard operator. \item If
$X_T=X$, then the set $X'$ in 6 coincides with $X$. Hence it is
concluded from 6 that $T$ is a weakly Picard operator.
\end{enumerate}
\end{proof}

\begin{cor} Let $X$ be sequentially complete and
satisfy Property $(\ast)$. Then the following statements are
equivalent:
\begin{description}
\item[\rm i)] $G$ is weakly connected; \item[\rm ii)] Every Banach
$G$-contraction $T:X\rightarrow X$ with $X_T\neq\emptyset$ is a
Picard operator; \item[\rm iii)] Every Banach $G$-contraction has
at most one fixed point in $X$.
\end{description}
\end{cor}

\begin{proof}
(i $\Rightarrow$ ii) It is an immediate consequence of Assertion 2
of Theorem \ref{main2}.\par (ii $\Rightarrow$ iii) Let
$T:X\rightarrow X$ be a Banach $G$-contraction. If
$X_T=\emptyset$, then $T$ is fixed point free since
$\fix(T)\subseteq X_T$. Otherwise, if $X_T\neq\emptyset$, then it
follows by the hypotheses that $T$ is a Picard operator and so it
has a unique fixed point. Therefore, $T$ has at most one fixed
point in $X$.\par (iii $\Rightarrow$ i) It follows immediately
from Theorem \ref{main1}.
\end{proof}

Setting $G=G_1$ (respectively, $G=G_2$) in Theorem \ref{main2}, we
get the following results in partially ordered uniform spaces.

\begin{cor} Let $X$ be sequentially complete
and $\preceq$ be a partial order on $X$ such that $X$ has the
following property:
\begin{description}
\item[]\hspace{5.5mm} If $\{x_n\}$ is a nondecreasing sequence in
$X$ converging to some $x$, then there exists a subsequence
$\{x_{n_k}\}$ of $\{x_n\}$ such that $x_{n_k}\preceq x$ for all
$k\geq1$.
\end{description}
Then an order Banach contraction $T:X\rightarrow X$ has a fixed
point in $X$ if and only if there exists an $x_0\in X$ such that
$x_0\preceq Tx_0$.
\end{cor}

\begin{cor} Let $X$ be sequentially complete
and $\preceq$ be a partial order on $X$ such that $X$ has the
following property:
\begin{description}
\item[]\hspace{5.5mm} If $\{x_n\}$ is a sequence in $X$ converging
to some $x$ whose consecutive terms are pairwise comparable, then
there exists a subsequence $\{x_{n_k}\}$ of $\{x_n\}$ such that
$x_{n_k}$ is comparable to $x$ for all $k\geq1$.
\end{description}
Then an order Banach contraction $T:X\rightarrow X$ has a fixed
point in $X$ if and only if there exists an $x_0\in X$ such that
$x_0$ is comparable to $Tx_0$.
\end{cor}

Now we are going to introduce two new types of continuity of
mappings from $X$ into itself. The idea of this definition is
taken from \cite[Definitions 2.2 and 2.4]{jac}.

\begin{defn}\rm Let $T:X\rightarrow X$ be an arbitrary mapping. We
say that
\begin{description}
\item[\rm i)] $T$ is orbitally continuous on $X$ whenever
$T^{p_n}x\rightarrow y$ implies $T(T^{p_n}x)\rightarrow Ty$ as
$n\rightarrow\infty$ if $x,y\in X$ and $\{p_n\}$ is a sequence of
positive integers. \item[\rm ii)] $T$ is orbitally $G$-continuous
on $X$ whenever $T^{p_n}x\rightarrow y$ implies
$T(T^{p_n}x)\rightarrow Ty$ as $n\rightarrow\infty$ if $x,y\in X$
and $\{p_n\}$ is a sequence of positive integers with
$(T^{p_n}x,T^{p_{n+1}}x)\in E(G)$ for $n=1,2,\ldots$ .
\end{description}
\end{defn}

It is clear that continuity implies orbital continuity and orbital
continuity implies orbital $G$-continuity for any graph $G$. In
the next example, we see that the converses of these statements
are not true.

\begin{exm}\rm Suppose that the set $X=[0,+\infty)$ is endowed with the uniformity
induced by the Euclidean metric.
\begin{enumerate}
\item Define a mapping $T$ from $X$ into itself by $Tx=1$ if
$x\neq0$ and $T0=0$. Then it is obvious that $T$ is not continuous
on $X$. To see the orbital continuity of $T$ on $X$, let $x,y\in
X$ and $\{p_n\}$ be a sequence of positive integers such that
$T^{p_n}x\rightarrow y$. If $x=0$, then $\{T^{p_n}x\}$ is the
constant sequence zero and hence $y=0$. Otherwise, if $x\neq0$,
then $\{T^{p_n}x\}$ is the constant sequence $1$ and hence $y=1$.
Therefore, in both cases, we have $T(T^{p_n}x)\rightarrow Ty$.
\item Next, endow $X$ with the graph $G_3$ defined by $V(G_3)=X$
and $E(G_3)=\Delta(X)$. Then the mapping $T:X\rightarrow X$ with
the rule
$$Tx=\left\{\begin{array}{cc}
\displaystyle\frac{x^2}2&x\neq0\\\\
1&x=0
\end{array}\right.$$
for all $x\in X$ is orbitally $G$-continuous on $X$. For, let
$x,y\in X$ and $\{p_n\}$ be an arbitrary sequence of positive
integers such that $T^{p_n}x\rightarrow y$ and
$(T^{p_n}x,T^{p_{n+1}}x)\in E(G)$ for $n=1,2,\ldots\,$. Then
$\{T^{p_n}x\}$ is a constant sequence, that is,
$$T^{p_1}x=\cdots=T^{p_n}x=\cdots=y.$$
Hence $T(T^{p_n}x)=Ty\rightarrow Ty$. On the other hand, if
$x=y=0$ and $p_n=n$ for $n=1,2,\ldots\,$, then we see that
$T^{p_n}x\rightarrow y$, whereas $T(T^{p_n})x\nrightarrow Ty$. So
$T$ fails to be orbitally continuous on $X$.
\end{enumerate}
\end{exm}

Using these new notions of continuity, we give some another
sufficient conditions for a Banach $G$-contraction to have a fixed
point and also to be a Picard or a weakly Picard operator when the
uniform space $X$ does not necessarily have Property $(\ast)$.

\begin{thm}\label{main3}
Let $X$ be complete and $T:X\rightarrow X$ be an orbitally
continuous Banach $G$-contraction. Then the following assertions
hold:
\begin{enumerate}
\item If $x\in X$ is such that $Tx\in[x]_{\widetilde G}$, then
there exists a unique $x^*\in \fix(T)$ such that $T^ny\rightarrow
x^*$ for all $y\in[x]_{\widetilde G}$. In particular, if
$x^*\in[x]_{\widetilde G}$, then $T\mid_{[x]_{\widetilde G}}$ is a
Picard operator. \item If $G$ is weakly connected, then $T$ is a
Picard operator. \item $T$ has a fixed point in $X$ if and only if
there exists an $x\in X$ such that $Tx\in[x]_{\widetilde G}$.
\item $T\mid_{X''}$ is a weakly Picard operator, where
$X''=\bigcup\{[x]_{\widetilde G}:Tx\in[x]_{\widetilde G}\}$. \item
If $Tx\in[x]_{\widetilde G}$ for all $x\in X$, then $T$ is a
weakly Picard operator.
\end{enumerate}
\end{thm}

\begin{proof}
\begin{enumerate}
\item Let $x\in X$ be such that $Tx\in[x]_{\widetilde G}$. Then it
follows by Proposition \ref{proposition} that the sequence
$\{T^nx\}$ is Cauchy in $X$ and so by the sequential completeness
of $X$, there exists an $x^*\in X$ such that $T^nx\rightarrow
x^*$. By the orbital continuity of $T$ on $X$, it follows that
$T^{n+1}x\rightarrow Tx^*$, and since $X$ is separated, we have
$x^*=Tx^*$, that is, $x^*\in\fix(T)$. Now, if $y\in[x]_{\widetilde
G}$ is arbitrary, then using Proposition \ref{proposition} once
more, we see that the sequences $\{T^nx\}$ and $\{T^ny\}$ are
Cauchy equivalent and hence an argument similar to that appeared
in the proof of (i $\Rightarrow$ ii) in Corollary \ref{corollary}
establishes that $T^ny\rightarrow x^*$. The uniqueness of $x^*$
follows immediately from the uniqueness of limits of convergent
sequences in separated uniform spaces. \item Take any arbitrary
$x\in X$. Because $G$ is weakly connected, we have
$[x]_{\widetilde G}=X$; in particular, $Tx\in[x]_{\widetilde G}$.
So by 1, $T$ is a Picard operator. \item If $x\in X$ is a fixed
point for $T$, then $Tx=x\in[x]_{\widetilde G}$. The converse
follows immediately from 1. \item If there is no $x\in X$ such
that $Tx\in[x]_{\widetilde G}$, then $X''$ is empty and vice
versa. So, in this case, there is nothing to prove. But if $x\in
X''$, then there exists an $x_0\in X$ such that both $Tx_0$ and
$x$ belong to $[x_0]_{\widetilde G}$. Now, by 1, we see that the
sequence $\{T^nx\}$ converges to a fixed point of $T$ and hence
$T\mid_{X''}$ is a weakly Picard operator. \item If
$Tx\in[x]_{\widetilde G}$, then the set $X''$ in 4 coincides with
$X$. Hence, by 4, $T$ is a weakly Picard operator.
\end{enumerate}
\end{proof}

The next example shows that the fixed point $x^*$ in Assertion 1
of Theorem \ref{main3} does not necessarily belong to
$[x]_{\widetilde G}$, i.e., the mapping $T\mid_{[x]_{\widetilde
G}}$ need not be a Picard operator. So if $X$ does not have
Property $(\ast)$, it seems that we cannot improve Theorem
\ref{main3} by formulating Assertions 3 and 5 of Theorem
\ref{main2} in a similar way.

\begin{exm}\rm Consider $\Bbb{R}$ with the uniformity induced by
the Euclidean metric and define a partial order $\preceq$ by
$$x\preceq y\ \Longleftrightarrow\ \Big(x=y\ \vee\
\big(x,y\in[1,4]\setminus\big\{\frac52\big\},\ x\leq
y\big)\Big)\qquad(x,y\in\Bbb{R}).$$ Define a mapping
$T:\Bbb{R}\rightarrow\Bbb{R}$ by
$$Tx=\left\{\begin{array}{cc}
2x&x<1\\\\
\displaystyle\frac{x+5}3&1\leq x\leq 4\\\\
2x-5&x>4
\end{array}\right.$$
for all $x\in\Bbb{R}$. Then it is easy to see that $\Bbb{R}$ is
(sequentially) complete and $T$ is a(n) (orbitally) continuous
Banach $G_1$-contraction on $\Bbb{R}$ for each contractive
constant $\alpha\in[\frac13,1)$. Now,
$$T4=3\in[1,4]\setminus\big\{\frac52\big\}=[4]_{{\widetilde
G}_1},$$ and $\fix(T)=\{0,\frac52,5\}$, which does not intersect
$[4]_{{\widetilde G}_1}$.
\end{exm}

\begin{cor}\label{cormain3} If $X$ is sequentially complete,
then the following statements are equivalent:
\begin{description}
\item[\rm i)] $G$ is weakly connected; \item[\rm ii)] Each
orbitally continuous Banach $G$-contraction on $X$ is a Picard
operator; \item[\rm iii)] Each orbitally continuous Banach
$G$-contraction has at most one fixed point in $X$.
\end{description}
In particular, if $\widetilde G$ is disconnected, then there
exists an orbitally continuous Banach $G$-contraction on $X$ with
at least two fixed points.
\end{cor}

\begin{proof} (i $\Rightarrow$ ii) It follows immediately from Assertion 2 of Theorem
\ref{main3}.\par (ii $\Rightarrow$ iii) It is obvious since each
Picard operator has exactly one fixed point.\par (iii
$\Rightarrow$ i) According to the proof of (iii $\Rightarrow$ i)
in Theorem \ref{main1}, it suffices to show that the mapping
$T:X\rightarrow X$ defined there is orbitally continuous on $X$.
To this end, let two points $x,y\in X$ and a sequence $\{p_n\}$ of
positive integers be given such that $T^{p_n}x\rightarrow y$. Then
the sequence $\{T^{p_n}x\}$ is either the constant sequence $x_0$
or the constant sequence $y_0$. If the former holds, then we have
$y=x_0$ since $X$ is separated. Hence
$$T(T^{p_n}x)=Tx_0=x_0\rightarrow x_0=Tx_0=Ty.$$
But if the latter holds, then a similar argument shows that
$T(T^{p_n}x)\rightarrow Ty$. Therefore, the mapping $T$ is
orbitally continuous on $X$.
\end{proof}

There is another version of Theorem \ref{main3}, where the orbital
continuity of $T$ is replaced with the weaker condition of orbital
$G$-continuity and one hypothesis of Theorem \ref{main3} is
changed a bit.

\begin{thm}\label{main4}
Let $X$ be sequentially complete and $T:X\rightarrow X$ be an
orbitally $G$-continuous Banach $G$-contraction. Then the
following assertions hold:
\begin{enumerate}
\item If $x\in X_T$, then there exists a unique $x^*\in\fix(T)$
such that $T^ny\rightarrow x^*$ for all $y\in[x]_{\widetilde G}$.
In particular, if $x^*\in[x]_{\widetilde G}$, then
$T\mid_{[x]_{\widetilde G}}$ is a Picard operator. \item If $X_T$
is nonempty and $G$ is weakly connected, then $T$ is a Picard
operator. \item $T$ has a fixed point in $X$ if and only if $X_T$
is nonempty. \item $T\mid_{X'}$ is a weakly Picard operator, where
$X'=\bigcup\{[x]_{\widetilde G}:x\in X_T\}$. \item If $X_T=X$,
then $T$ is a weakly Picard operator.
\end{enumerate}
\end{thm}

\begin{proof}
\begin{enumerate}
\item Let $x\in X_T$. Then $Tx\in[x]_{\widetilde G}$ and it
follows by Proposition \ref{proposition} that the sequence
$\{T^nx\}$ is Cauchy in $X$. Because $X$ is sequentially complete,
there exists an $x^*\in X$ such that $T^nx\rightarrow x^*$. Since
$(x,Tx)\in E(G)$ and $T$ preserves the edges of $G$, we have
$(T^nx,T^{n+1}x)\in E(G)$ for all $n\geq1$ and it follows by the
orbital $G$-continuity of $T$ on $X$ that $T^{n+1}x\rightarrow
Tx^*$. So $x^*=Tx^*$ because $X$ is separated. Now, take any
arbitrary $y\in[x]_{\widetilde G}$. Then using Proposition
\ref{proposition} again, it is concluded that the sequences
$\{T^nx\}$ and $\{T^ny\}$ are Cauchy equivalent and an argument
similar to that appeared in the proof of (i $\Rightarrow$ ii) in
Corollary \ref{corollary} establishes that $T^ny\rightarrow x^*$.
The uniqueness of $x^*$ follows immediately from the uniqueness of
limits of convergent sequences in separated uniform spaces. \item
Choose any $x\in X_T$. Since $G$ is weakly connected, we have
$[x]_{\widetilde G}=X$, and thus by 1, $T$ is a Picard operator.
\item If $x$ is a fixed point for $T$, then clearly $x\in X_T$.
The converse follows immediately from 1. \item If $X_T$ is empty
then so is $X'$ and vice versa. So, in this case, there is nothing
to prove. If $x\in X$, then there exists an $x_0\in X_T$ such that
$x\in[x_0]_{\widetilde G}$. So by 1, the sequence $\{T^nx\}$
converges to a fixed point of $T$. Whence $T\mid_{X'}$ is a weakly
Picard operator. \item If $X_T=X$, then the set $X'$ in 4
coincides with $X$. Therefore, by 4, $T$ is a weakly Picard
operator.
\end{enumerate}
\end{proof}

Similar to Corollary \ref{cormain3}, we have the following:

\begin{cor} If $X$ is sequentially complete, then the following
statements are equivalent:
\begin{description}
\item[\rm i)] $G$ is weakly connected; \item[\rm ii)] Each
orbitally $G$-continuous Banach $G$-contraction $T:X\rightarrow X$
with $X_T\neq\emptyset$ is a Picard operator; \item[\rm iii)] Each
orbitally $G$-continuous Banach $G$-contraction has at most one
fixed point in $X$.
\end{description}
In particular, if $\widetilde G$ is disconnected, then there
exists an orbitally $G$-continuous Banach $G$-contraction on $X$
with at least two fixed points.
\end{cor}

\begin{proof}
(i $\Rightarrow$ ii) It follows immediately from Assertion 2 of
Theorem \ref{main4}.\par (ii $\Rightarrow$ iii) Let
$T:X\rightarrow X$ be an orbitally $G$-continuous Banach
$G$-contraction. If $X_T=\emptyset$, then $T$ is fixed point free
since $\fix(T)\subseteq X_T$. Otherwise, if $X_T\neq\emptyset$,
then it follows by the hypothesis that $T$ is a Picard operator
and so it has a unique fixed point. Therefore, $T$ has at most one
fixed point in $X$.\par (iii $\Rightarrow$ i) According to the
proof of (iii $\Rightarrow$ i) in Corollary \ref{cormain3}, it is
obvious because each orbitally continuous mapping on $X$ is
orbitally $G$-continuous.
\end{proof}

Recall that a mapping $T$ from a uniform space $X$ (not
necessarily endowed with a graph) into itself is nonexpansive if
$(x,y)\in V$ implies $(Tx,Ty)\in V$ for all $x,y\in X$ and all
$V\in\mathscr V$. Clearly, each nonexpansive mapping
$T:X\rightarrow X$ is (uniformly) continuous on $X$. As the final
result of this paper, we give a sufficient condition for a
nonexpansive map whose some restriction is a Picard operator. We
first have the following lemma:

\begin{lem}\label{222} Let $A$ be a dense subset of $X$ and let
$T:X\rightarrow X$ be a mapping such that the family
$\{T^n:n\geq1\}$ is equicontinuous on $X$. If there exists an
$x^*\in X$ such that $T^nx\rightarrow x^*$ for all $x\in A$, then
$T^nx\rightarrow x^*$ for all $x\in X$.
\end{lem}

\begin{proof}
Let $x\in X$ and let an entourage $U$ of $X$ be given. Choose a
$V\in\mathscr U$ such that $V\circ V\subseteq U$. Because the
family $\{T^n:n\geq1\}$ is equicontinuous on $X$, there exist a
positive integer $N_1$ and an entourage $W$ of $X$ such that
$(T^nx,T^ny)\in V$ for all $n\geq N_1$ and all $y\in X$ with
$(x,y)\in W$. Since $A$ is dense in $X$, the set $W[x]\cap A$ is
nonempty, i.e., there exists a $y\in A$ such that $(x,y)\in W$. By
the hypothesis, we have $T^ny\rightarrow x^*$. So there exists an
$N_2>0$ such that $(T^ny,x^*)\in V$ for all $n\geq N_2$. Now, if
$n\geq \max\{N_1,N_2\}$, then we have
$$(T^nx,T^ny),(T^ny,x^*)\in V,$$
and so
$$(T^nx,x^*)\in V\circ V\subseteq U.$$
\end{proof}

\begin{thm} Let $X$ be sequentially complete and
let $T:X\rightarrow X$ be a Banach $G$-contraction. If $T$ is
nonexpansive or the family $\{T^n:n\geq1\}$ is equicontinuous on
$X$, then $T\mid_{\overline{[x]_{\widetilde G}}}$ is a Picard
operator for all $x\in X$ with $Tx\in[x]_{\widetilde G}$. In
particular, if $G$ is weakly connected, then each nonexpansive
Banach $G$-contraction is a Picard operator.
\end{thm}

\begin{proof} Because the nonexpansivity of $T$ implies the
equicontinuity of the family $\{T^n:n\geq1\}$ on $X$, we need only
to consider the case where the second condition holds. So let
$x\in X$ be such that $Tx\in[x]_{\widetilde G}$. Then it follows
by Proposition \ref{proposition} that $[x]_{\widetilde G}$ is
$T$-invariant. Now, given any $y\in\overline{[x]_{\widetilde G}}$,
there exists a net $\{x_\lambda:\lambda\in\Lambda\}$ such that
$x_\lambda\rightarrow y$. By continuity of $T$ on $X$, we get
$Tx_\lambda\rightarrow Ty$ and so $Ty\in\overline{[x]_{\widetilde
G}}$ since $\{Tx_\lambda:\lambda\in\Lambda\}$ is still a net in
$[x]_{\widetilde G}$. Therefore, $\overline{[x]_{\widetilde G}}$
is $T$-invariant. Moreover, by Assertion 1 of Theorem \ref{main3},
there exists a unique $x^*\in \fix(T)$ such that $T^ny\rightarrow
x^*$ for all $y\in[x]_{\widetilde G}$. Clearly,
$x^*\in\overline{[x]_{\widetilde G}}$ is the unique fixed point of
$T$ in $\overline{[x]_{\widetilde G}}$ and since the family
$\{T^n:n\geq1\}$ is equicontinuous on $X$, the result follows
immediately from Lemma \ref{222}.
\end{proof}

\end{document}